%% file: rep2groupsinBC2vect.tex
\documentclass{amsart}
\usepackage[utf8]{inputenc}

\usepackage{amsmath}
\usepackage{amssymb,amsfonts}
\usepackage[utf8]{inputenc}

\usepackage{xypic}

\usepackage[pdfborder={0 0 0}]{hyperref}
\usepackage{todonotes}
\usepackage[capitalize]{cleveref}
\usepackage{verbatim}

\newtheorem{theorem}{Theorem}

\newtheorem{lemma}{Lemma}
\newtheorem{proposition}{Proposition}

\newcommand{\Aa}{\mathcal{A}}
\newcommand{\Bb}{\mathcal{B}}
\newcommand{\Cc}{\mathcal{C}}

\newcommand{\Gg}{\mathcal{G}}

\newcommand{\Rr}{\mathcal{R}}

\newcommand{\Vv}{\mathcal{V}}

\newcommand{\EE}{\mathbb{E}}

\newcommand{\GG}{\mathbb{G}}

\newcommand{\Abf}{\mathbf{A}}
\newcommand{\Bbf}{\mathbf{B}}
\newcommand{\Cbf}{\mathbf{C}}

\newcommand{\To}{\rightarrow}
\newcommand{\coker}{\mathrm{coker}\,}
\newcommand{\im}{\mathrm{im}\,}

\title{On the representations of 2-groups in Baez-Crans 2-vector spaces}
\author{Benjam\'in A. Heredia}
\address{Departamento de Matem\'atica e Centro de Matem\'atica e Aplica\c{c}oes (CMA), FCT, Universidade Nova de Lisboa}
\email{b.heredia@fct.unl.pt}
\thanks{The first author is supported by the Funda\c{c}\~ao para a Ci\^encia e a Tecnologia (Portuguese Foundation for Science and Technology) through the project UID/MAT/00297/2013 (Centro de Matem\'atica e Aplica\c{c}\~oes). Also supported by the Generalitat de Catalunya (Project: 2014 SGR 634) and from by DGI of Spain, Project MTM2011-22554. The first author would like to thank professor Jo\~ao Faria Martins for introducing him into this subject, and also the kind hospitality during his stay at the Departament de Matem\`atiques, Universitat Polit\`ecnica de Catalunya.}

\author{Josep Elgueta}
\address{Departament de Matem\`atiques Universitat Polit\`ecnica de Catalunya}
\email{josep.elgueta@upc.edu}
\thanks{The second author acknowledges the financial support from the Generalitat de Catalunya (Project: 2014 SGR 634), the Ministerio de Econom\'ia y Competitividad of Spain (Project: MTM2015-69135-9) and from the Centro de Matem\'atica e Aplica\c{c}oes of the Universidade Nova de Lisboa (Project: UID/MAT/00297/2013), Portugal, and the kind hospitality during his stay at the Centro de Matem\'atica e Aplica\c{c}oes.}



\begin{document}

\maketitle

\begin{abstract}
    We prove that the theory of representations of a finite 2-group $\GG$ in Baez-Crans 2-vector spaces over a field $k$ of characteristic zero essentially reduces to the theory of $k$-linear representations of the group of isomorphism classes of objects of $\GG$, the remaining homotopy invariants of $\GG$ playing no role. It is also argued that a similar result is expected to hold for {\em topological} representations of compact topological 2-groups in suitable {\em topological} Baez-Crans 2-vector spaces.
\end{abstract}

\section{Introduction}

In the last two decades there have been a few attempts to generalize the representation theory of groups to the higher dimensional setting of categories. See Baez {\em et al} \cite{BBFW2012}, Bartlett \cite{bB2011}, Crane and Yetter \cite{CY2005}, Elgueta \cite{jE2007, jE2011}, Ganter and Kapranov \cite{GK2008}, and Ganter \cite{nG2015}. 
By analogy with the classical setting, it is natural to try to represent 2-groups in a suitable categorification of the category $\Vv ect_k$ of (finite dimensional) vector spaces over a ground field $k$, often called the 2-category of {\em 2-vector spaces} over $k$. 

One of the first proposals of definition of 2-vector space is that of Baez and Crans \cite{BC2004}. According to these authors, a 2-vector space over $k$ is an internal category in $\Vv ect_k$, and they proved that this is the same thing as a 2-term chain complex of vector spaces over $k$, i.e. a $k$-linear map $d:V_1\to V_0$. To our knowledge, the unique existing work on the representation theory of 2-groups in these 2-vector spaces is the very preliminary presentation by Forrester-Barker \cite{FB2003}.

The purpose of this short paper is to show that the representation theory of a finite 2-group $\GG$ in Baez-Crans 2-vector spaces over a field of characteristic zero is in some sense trivial. More precisely, it will be shown that the homotopy category of the corresponding 2-category of representations of $\GG$ is simply equivalent to the product category $\Rr ep_k(\pi_0(\GG))\times\Rr ep_k(\pi_0(\GG))$, where $\pi_0(\GG)$ is the group of isomorphism classes of objects of $\GG$, and $\Rr ep_k(\pi_0(\GG))$ is the category of $k$-linear representations of $\pi_0(\GG)$. In particular, the remaining homotopy invariants of $\GG$ classifying it up to equivalence, namely, the abelian group $\pi_1(\GG)$ of automorphisms of the unit object of $\GG$ and the cohomology class in $H^3(\pi_0(\GG),\pi_1(\GG))$, play no role. We also argue that a similar result will be true for compact topological 2-groups and their topological representations in a suitable 2-category of {\em topological} Baez-Crans 2-vector spaces. The result is basically a consequence of the fact that the underlying category of the 2-category of Baez-Crans 2-vector spaces over $k$ is essentially $\Vv ect_k\times\Vv ect_k$. 

To avoid writing a too long paper, we will assume the reader is familiar with the notions of 2-group and 2-category, and with the corresponding notions of morphism, which are understood in the weak sense, including the notions of pseudonatural transformation and modification. We refer the reader to Leinster \cite{L1998} or Borceux \cite{fB1994} for an introduction to 2-categories, and to Baez and Lauda \cite{BL2004} for an introduction to 2-groups.

{\bf Notation.} We will use letters like $\Aa,\Bb,\Cc,...$ to denote categories, and $\Abf,\Bbf,\Cbf,...$ to denote 2-categories. Vertical composition of 2-cells will be denoted by juxtaposition, and composition of 1-cells and horizontal composition of 2-cells by $\circ$.

\input{2chains2}

\input{repr_2-cat2}

\bibliography{library}{}
\bibliographystyle{plain}

\end{document}

%% file: 2chains2.tex
\section{The 2-category of Baez-Crans 2-vector spaces}

Let us start by describing the 2-category $\mathbf{Ch}_2(\Aa)$ of 2-term chain complexes (i.e. chain complexes concentrated in degrees 1 and 0) in any abelian category $\Aa$. We will be mainly concerned with the case $\Aa=\Vv ect_k$, the category of finite dimensional vector spaces over a field $k$. $\mathbf{Ch}_2(k)$ is short notation for $\mathbf{Ch}_2(\Vv ect_k)$. We will refer to $\mathbf{Ch}_2(k)$ as the 2-category of Baez-Crans 2-vector spaces over $k$.

\subsection{}
An object of \(\mathbf{Ch}_2(\mathcal{A})\) is a morphism of \(\mathcal{A}\), that is \(d_V=d:V_1\to V_0\), denoted by \(V_\bullet\). The morphism \(d\) is called the differential.

A 1-cell \(f_\bullet= (f_1,f_0):V_\bullet \to W_\bullet\) is a commutative square
\[
\xymatrix{
V_1 \ar[r]^{d_V}
  \ar[d]_{f_1}
&
V_0\ar[d]^{f_0}
\\
W_1\ar[r]^{d_W}
&
W_0,
}
\]
and a 2-cell \(\sigma:f_\bullet \Rightarrow g_\bullet:V_\bullet \to W_\bullet\) is a morphism \(\sigma:V_0\to W_1\) in \(\mathcal{A}\) such that
\[
\begin{array}{l}
d_W\circ \sigma = g_0 - f_0 \\
\sigma \circ d_V = g_1 - f_1.
\end{array}
\]
The composition of 1-cells is given by the composition in \(\mathcal{A}\), that is, given \((f_1,f_0):U_\bullet \to V_\bullet\) and \((g_1,g_0):V_\bullet \to W_\bullet\) the composite is
\[
(g_1,g_0)\circ (f_1,f_0)=(g_1\circ f_1,g_0\circ f_0):U_\bullet \to W_\bullet,
\]
and the identity morphisms are given by \(1_{V_\bullet}=(1_{V_1},1_{V_0})\).

The vertical composite of \(\sigma:f_\bullet \Rightarrow g_\bullet\) and \(\tau:g_\bullet \Rightarrow h_\bullet\) is given by the addition in $\Aa$, that is
\[
\tau\sigma=\tau + \sigma:f_\bullet\Rightarrow h_\bullet,
\]
while horizontal composite of \(\sigma:f_\bullet \Rightarrow g_\bullet:U_\bullet \to V_\bullet\) and \(\sigma':f'_\bullet \Rightarrow g'_\bullet:V_\bullet \to W_\bullet\) is given by the map
\begin{eqnarray*}
\sigma'\circ\sigma &=f'_1\circ\sigma + \sigma'\circ g_0 \\ &= g'_1\circ\sigma + \sigma'\circ f_0.
\end{eqnarray*}
Finally, identity 2-cells are given by \(1_{f_\bullet}=0:V_0\to W_1\) for any 1-cell \(f_\bullet:V_\bullet \to W_\bullet\). In particular, whiskerings are given by
\[
\sigma\circ 1_{f_\bullet}=\sigma\circ f_0,
\quad
1_{f'_\bullet}\circ\sigma=f'_1\circ\sigma.
\]
It is straightforward to check that \(\mathbf{Ch}_2(\mathcal{A})\) is a strict 2-category. In fact, it is a category enriched in groupoids. Each 2-cell $\tau$ is invertible with inverse $-\tau$.

\subsection{}
The next two results will be needed later. Both hold in an arbitrary abelian category $\Aa$.

\begin{lemma}\label{lema_zero_objects}
  A 2-term chain complex \(d:V_1\to V_0\) is equivalent in $\mathbf{Ch}_2(\Aa)$ to the zero complex $0_\bullet=0\to 0$ if and only if the differential \(d\) is an isomorphism in $\Aa$.
\end{lemma}

    \begin{proof}
      It readily follows from the definitions that $V_\bullet\simeq 0_\bullet$ if and only if there exists a 2-cell $1_{V_\bullet}\Rightarrow 0_{V_\bullet}$, and this happens if and only if $d$ is an isomorphism.
    \end{proof}
\begin{lemma}\label{lema_d_oplus_1}
  For any object $W$ of $\Aa$ and any object $V_\bullet$ of $\mathbf{Ch}_2(\Aa)$, the 2-term chain complexes $d:V_1\to V_0$ and $d\oplus 1_W:V_1\oplus W\to V_0\oplus W$ are equivalent in $\mathbf{Ch}_2(\Aa)$.
\end{lemma}

    \begin{proof}
      Let $\pi_i:V_i\oplus W\to V_i$, $\iota_i:V_i\to V_i\oplus W$ be the canonical projections and injections for $i=0,1$. Then the 1-cell $\iota_\bullet=(\iota_0,\iota_1):V_\bullet\to V_\bullet\oplus W$ is an equivalence with $\pi_\bullet=(\pi_0,\pi_1):V_\bullet\oplus W\To V_\bullet$ as a pseudoinverse. Indeed, $\pi_\bullet\circ\iota_\bullet=1_{V_\bullet}$ while $\iota_\bullet\circ\pi_\bullet\cong 1_{V_\bullet\oplus W}$ via the 2-isomorphism $0\oplus 1_W:V_0\oplus W\to V_1\oplus W$.
    \end{proof}

\subsection{}
Let $\mathbf{Ch}'_2(\Aa)$ be the full sub-2-category of $\mathbf{Ch}_2(\Aa)$ with objects the zero morphisms $0:V_1\to V_0$ in $\Aa$. The significance of $\mathbf{Ch}'_2(\Aa)$ comes from the fact that all objects in $\mathbf{Ch}_2(\Aa)$ are equivalent to an object in $\mathbf{Ch}'_2(\Aa)$ when $\Aa$ is such that each short exact sequence splits, for instance when $\Aa$ is $\Vv ect_k$. Such an $\Aa$ will be called a {\em split abelian} category. More precisely, we have the following result, already implicit in \cite[Proposition 305]{mD2008}.

\begin{proposition}\label{zero_morphisms}
  Let $\Aa$ be a split abelian category. Then $\mathbf{Ch}'_2(\Aa)$ is biequivalent to $\mathbf{Ch}_2(\Aa)$.
\end{proposition}

    \begin{proof}
      It is enough to see that each object of $\mathbf{Ch}_2(\Aa)$ is equivalent to a zero morphism in $\Aa$. In fact, an object $d:U_1\to U_0$ of $\mathbf{Ch}_2(\Aa)$ is equivalent to the zero morphism $\ker d\stackrel{0}{\to}\coker d$. Indeed, we have the short exact sequences
      \[
      \begin{array}{l}
      0\rightarrow \ker d\rightarrow U_1\rightarrow \coker(\ker d)\rightarrow 0
      \\
      0\rightarrow \ker(\coker d)\rightarrow U_0\rightarrow \coker d\rightarrow 0,
      \end{array}
      \]
      and $\coker(\ker d)\cong \ker(\coker d)$. As usual we identify both objects and denote them by $\im d$. It follows that we have a commutative square of the form
      \[
      \xymatrix{
      U_1\ar[r]^d\ar[d]_{\cong}
      &
      U_0\ar[d]^{\cong}
      \\
      \ker d\oplus \im d\ar[r]_{0\oplus 1}
      &
      \coker d\oplus \im d.
      }
     \]
     In particular, the top and the bottom morphisms are equivalent as objects in $\mathbf{Ch}_2(\Aa)$ (in fact, isomorphic). The result now follows from Lemma~\ref{lema_d_oplus_1}.
    \end{proof}

\subsection{}\label{Eq(V)}
It easily follows from the above description of $\mathbf{Ch}_2(\Aa)$ that the 2-group of self-equivalences of an object $V_\bullet$ in $\mathbf{Ch}'_2(\Aa)$ is the skeletal and strict 2-group that has the elements of $\mathrm{Aut}_\Aa(V_1)\times\mathrm{Aut}_\Aa(V_0)$ as objects, and the elements of $\Aa(V_0,V_1)\times\mathrm{Aut}_\Aa(V_1)\times\mathrm{Aut}_\Aa(V_0)$ as morphisms, with $(\sigma,f_1,f_0):(f_1,f_0)\to(f_1,f_0)$. The composition of morphisms is given by the sum in $\Aa(V_0,V_1)$, and the tensor product is given on objects and morphisms by
    \begin{align*}
       (f'_1,f'_0)\otimes(f_1,f_0)&=(f'_1\circ f_1,f'_0\circ f_0)
       \\
       (\sigma',f'_1,f'_0)\otimes(\sigma,f_1,f_0)&=(\sigma'\circ f_0+f'_1\circ\sigma,f'_1\circ f_1,f'_0\circ f_0).
    \end{align*}


%% file: repr_2-cat2.tex
\section{Representations in Baez-Crans 2-vector spaces}
\label{sec:repr_2-cat}

From now on, we will assume that the ground field $k$ is of characteristic zero. The goal of this section is to prove that the representation theory of a finite 2-group $\GG$ in the 2-category of Baez-Crans 2-vector spaces is trivial in the sense made precise below. More generally, this is true for representations in $\mathbf{Ch}_2(\Aa)$ for any split $k$-linear abelian category $\Aa$. 

We start by describing the 2-category $\mathbf{Rep}_{\mathbf{Ch}_2(\Aa)}(\GG)$ of representations of $\GG$ in $\mathbf{Ch}_2(\Aa)$ for any $\Aa$. For later use, we do it for an arbitrary $\Aa$, and we next focus on the split case.

\subsection{Description of the generic 2-category of representations}
Let $\GG[1]$ be the one-object 2-groupoid with $\GG$ as 2-group of self-equivalences of the unique object. By definition,  $\mathbf{Rep}_{\mathbf{Ch}_2(\Aa)}(\GG)$ is the (strict) 2-category of pseudofunctors from $\GG[1]$ to $\mathbf{Ch}_2(\Aa)$, pseudonatural transformations between them, and modifications between these. When unpacked, this definition leads to the 2-category with the following 0-, 1- and 2-cells.

\subsubsection{}\label{objects}
An object in $\mathbf{Rep}_{\mathbf{Ch}_2(\Aa)}(\GG)$ is given by the following data:
\begin{itemize}
    \item[(O1)] a 2-term chain complex $d:V_1\to V_0$, also denoted by $V_\bullet$;
    \item[(O2)] for each object $a\in\Gg$ a pair $f^a_\bullet=(f^a_1,f^a_0)$ which makes the square in $\Aa$
        \[
        \xymatrix{
            V_1\ar[r]^d
               \ar[d]_{f^a_1}
            &
            V_0\ar[d]^{f^a_0}
            \\
            V_1\ar[r]_d 
            &
            V_0
        }
        \]
    commute;
    \item[(O3)] for each morphism $\phi:a\to a'$ in $\Gg$ a morphism $\tau_\phi:V_0\to V_1$ in $\Aa$ such that
\[
\begin{array}{l}
f^{a'}_0-f^a_0=d\circ\tau_\phi,
\\
f^{a'}_1-f^a_1=\tau_\phi\circ d;
\end{array}
\]
    \item[(O4)] for each pair of objects $a,b\in\Gg$ a morphism $\tau_{a,b}:V_0\to V_1$ in $\Aa$ such that
\[
\begin{array}{l}
f^{a\otimes b}_0-f^a_0\circ f^b_0=d\circ\tau_{a,b},
\\
f^{a\otimes b}_1-f^a_1\circ f^b_1=\tau_{a,b}\circ d.
\end{array}
\]
    \item[(O5)] a morphism $\tau_e:V_0\to V_1$ such that
\[
\begin{array}{l}
f^{e}_0-1_{V_0}=d\circ\tau_{e},
\\
f^{e}_1-1_{V_1}=\tau_{e}\circ d.
\end{array}
\]
\end{itemize}
Moreover, these data must satisfy the following axioms:
\begin{itemize}
    \item[(AO1)] $\tau_{\phi'\circ\phi}=\tau_{\phi'}+\tau_{\phi}$ for every composable 1-cells $a\stackrel{\phi}{\to}a'\stackrel{\phi'}{\to}a''$ in $\Gg$;
    \item[(AO2)] $\tau_{1_a}=0$ for each object $a\in\Gg$;
    \item[(AO3)] $\tau_{a,b}+f^b_1\circ\tau_\phi+\tau_\psi\circ f^{a'}_0=\tau_{\phi\otimes\psi}+\tau_{a',b'}$ for every 1-cells $\phi:a\to a'$ and $\psi:b\to b'$ in $\Gg$;
    \item[(AO4)] $\tau_{\alpha_{a,b,c}}+\tau_{a,b\otimes c}+f^a_1\circ \tau_{b,c}=\tau_{a\otimes b,c}+\tau_{a,b}\circ f^c_0$ for every objects $a,b,c\in\Gg$;
    \item[(AO5)] $\tau_{a,e}+f^a_1\circ\tau_e=\tau_{\rho_a}$ and $\tau_{e,a}+\tau_e\circ f^a_0=\tau_{\lambda_a}$ for each object $a\in\Gg$.
\end{itemize}
Axioms (AO1)-(AO2) correspond to the functoriality of the assignments $\phi\mapsto\tau_\phi$, axiom (AO3) to the naturality of $\tau_{a,b}$ in $a,b$ and (AO4)-(AO5) to the coherence conditions. We will denote such an object by $(V_\bullet,\{f^a\},\{\tau_\phi\},\{\tau_{a,b}\},\tau_e)$ or just $V_\bullet$ when the action of $\GG$ on $V_\bullet$ is implicitly understood.

\subsubsection{}\label{1cells}
Given objects $(U_\bullet,\{f^a_\bullet\},\{\tau_\phi\},\{\tau_{a,b}\},\tau_e)$ and $(V_\bullet,\{g^a_\bullet\},\{\sigma_\phi\},\{\sigma_{a,b}\},\sigma_e)$, a 1-cell or 1-{\em intertwiner} from the first to the second consists of the following data:
\begin{itemize}
    \item[(I1)] a pair $r_\bullet=(r_1,r_0)$ which makes the square in $\Aa$
    \[
    \xymatrix{
    U_1\ar[r]^{d_U}
        \ar[d]_{r_1}
    &
    U_0\ar[d]^{r_0}
    \\
    V_1\ar[r]_{d_V}
    &
    V_0}
    \]
    commute;
    \item[(I2)] for each $a\in\Gg$ a morphism $\mu_a:U_0\to V_1$ in $\Aa$ such that
    \[
    \begin{array}{l}
    g^{a}_1\circ r_1-r_1\circ f^a_1=\mu_a\circ d_U,
    \\
    g^{a}_0\circ r_0-r_0\circ f^a_0=d_V\circ\mu_a;
    \end{array}
    \]
\end{itemize}
Moreover, these data must satisfy the following axioms:
\begin{itemize}
    \item[(AI1)] $\mu_a+\sigma_\phi\circ r_0=r_1\circ\tau_\phi+\mu_{a'}$ for each morphism $\phi:a\to a'$ in $\Gg$;
    \item[(AI2)] $r_1\circ\tau_{a,b}+\mu_{a\otimes b}=\mu_a\circ f^b_0+g^a_1\circ\mu^b+\sigma_{a,b}\circ r_0$ for every objects $a,b\in\Gg$;
    \item[(AI3)] $\mu_e+r_1\circ\tau_e=\sigma_e\circ r_0$.
\end{itemize}
Axiom (AI1) corresponds to the naturality of $\mu_a$ in $a$, and axioms (AI2)-(AI3) to the coherence conditions. We will denote such a 1-cell by $(r_\bullet,\mu)$ or just $r_\bullet$.

\subsubsection{}\label{2-cells}
Finally, given 1-cells $(r_\bullet,\mu)$, $(s_\bullet,\nu)$ between two representations $U_\bullet$ and $V_\bullet$, a 2-cell or {\em 2-intertwiner} from the first 1-cell to the second consists of morphism $\omega:U_0\to V_1$ such that
    \[
    \begin{array}{l}
    s_1-r_1=\omega\circ d_U,
    \\
    s_0-r_0=d_V\circ\omega
    \end{array}
    \]
and satisfying the following naturality axiom:
\begin{itemize}
    \item[(A2I)] $g^a_1\circ\omega+\mu_a=\nu_a+\omega\circ f^a_0$ for each object $a\in\Gg$.
\end{itemize}

\subsection{Case of a split $k$-linear abelian category}
Without loss of generality, we assume from now on that $\GG$ is the strict skeletal 2-group
    \[
    \GG=\pi_1[1]\rtimes_z\pi_0[0]
    \]
for some group $\pi_0$, left $\pi_0$-module $\pi_1$, and normalized 3-cocycle $z:\pi_0^3\to\pi_1$. This means that $\GG$ has the elements of $\pi_0$ as objects, and the pairs $(a,g)\in\pi_1\times\pi_0$, with $(a,g):g\to g$, as morphisms. Moreover, composition is given by the sum in $\pi_1$, and the tensor product by the product in $\pi_0$ on objects and by
\[
(a,g)\otimes(a',g')=(a+g\rhd a',gg')
\]
on morphisms ($\rhd$ stands for the left action of $\pi_0$ on $\pi_1$). Finally, the associator is given by $\alpha_{g,g',g''}=(z(g,g',g''),gg'g'')$ and the left and right unit isomorphisms are trivial. By Sinh's theorem \cite{hxS1975}, any 2-group is of this type up to equivalence (see also Baez and Lauda \cite{BL2004}).

For example, the 2-group $\EE q(V_\bullet)$ of self-equivalences of an object $V_\bullet$ in $\mathbf{Ch}'_2(\Aa)$ (see \S~\ref{Eq(V)}) is of this type, with
    \begin{align*}
        \pi_0&=\mathrm{Aut}_\Aa(V_1)\times\mathrm{Aut}_\Aa(V_0),
        \\
        \pi_1&=\Aa(V_0,V_1),
    \end{align*}
the left action of $\pi_0$ on $\pi_1$ given by
    \[
    (f_1,f_0)\rhd\sigma=f_1\circ\sigma\circ f^{-1}_0,
    \]
and the 3-cocycle $z$ equal to zero.

\subsubsection{}
We are interested in the 2-category $\mathbf{Rep}_{\mathbf{Ch}_2(\Aa)}(\GG)$ when $\Aa$ is split. Now, for these abelian categories $\mathbf{Rep}_{\mathbf{Ch}_2(\Aa)}(\GG)$ is biequivalent to $\mathbf{Rep}_{\mathbf{Ch}'_2(\Aa)}(\GG)$. This is a consequence of Proposition~\ref{zero_morphisms} and the general fact that for any biequivalent 2-categories $\Cbf,\Cbf'$ the corresponding representation 2-categories $\mathbf{Rep}_\Cbf(\GG)$ and $\mathbf{Rep}_{\Cbf'}(\GG)$ are biequivalent. This allows us to restrict from now on to representations of $\GG$ in $\mathbf{Ch}_2'(\Aa)$.

\subsubsection{}\label{representation}
From the description in \S~\ref{objects} of a representation of $\GG$ in $\mathbf{Ch}_2(\Aa)$, it follows that a representation as self-equivalences of an object in $\mathbf{Ch}_2'(\Aa)$ amounts to the following data: 

\begin{enumerate}
    \item[(O1$'$)] two representations of $\pi_0$ in $\Aa$, denoted by $\rho_i:\pi_0\to\mathrm{Aut}_\Aa(V_i)$, $i=0,1$;
    \item[(O2$'$)] a morphism of (left) $\pi_0$-modules $\beta:\pi_1\to\Aa(V_0,V_1)_{\rho_1}^{\rho_0}$, where $\Aa(V_0,V_1)_{\rho_1}^{\rho_0}$ stands for the abelian group $\Aa(V_0,V_1)$ equipped with the $\pi_0$-action induced by the representations $\rho_0,\rho_1$, and
    \item[(O3$'$)] a normalized 2-cochain $c:\pi_0^2\to\Aa(V_0,V_1)_{\rho_1}^{\rho_0}$ such that $\partial c=\beta_\ast(z)$. 
\end{enumerate}
Such a description can also be obtained from the above description of $\EE q(V_\bullet)$, the fact that a representation of $\GG$ as self-equivalences of $V_\bullet$ is nothing but a 2-group homomorphism $\GG\to\EE q(V_\bullet)$, and the description of the homomorphisms between strict skeletal 2-groups. Moreover, it may be shown that by changing the 2-cochain $c$ by another one differing from $c$ by a coboundary gives a representation which is equivalent to the original one (see \cite[Theorem 2.7]{jE2007}). The representation so defined will be denoted by $(\rho_1,\rho_0,\beta,c)$.

\begin{proposition}\label{representations_up_to_equivalence}
  Let us assume that $\GG$ is finite (i.e. $\pi_0$ and $\pi_1$ are finite), and that $\Aa$ is $k$-linear, with $k$ a field of characteristic zero. Then for any representation $(\rho_1,\rho_0,\beta,c)$ we have:
  \begin{itemize}
      \item[(i)] $\beta=0$, and
      \item[(ii)] $(\rho_1,\rho_0,0,c)$ is equivalent to $(\rho_1,\rho_0,0,0)$.
  \end{itemize}
  In particular, up to equivalence a representation of $\GG$ in $\mathbf{Ch}'_2(\Aa)$ is completely given by two representations of $\pi_0$ in $\Aa$.
\end{proposition}

    \begin{proof}
      If $\Aa$ is $k$-linear, $\Aa(V_0,V_1)$ is a $k$-vector space. Since the charateristic of $k$ is zero, the underlying abelian group has no torsion and the only morphism of abelian groups $\beta:\pi_1\to\Aa(V_0,V_1)$ is $\beta=0$. In this case, the 2-cochain is a 2-cocycle, and the second statement follows then from the next lemma (applied to $G$-bimodules with trivial right action of $G$) together with the above mentioned fact that 2-cochains differing by a coboundary determine equivalent representations.
    \end{proof}

\begin{lemma}\label{finite_group_cohomology}
  Let $G$ be a finite group, and $V$ a $k$-vector space equipped with a structure of $G$-bimodule. Then $H^n(G,V)=0$ for all $n\geq 1$.
\end{lemma}

    \begin{proof}
      Let $z:G^n\to V$ be an $n$-cocycle, with $n\geq 1$. Then an $n$-cochain with boundary $z$ is given by the map $c:G^{n-1}\to V$ defined by
      \[
      c(g_1,...,g_{n-1})=\frac{(-1)^n}{|G|}\sum_{k\in G}z(g_1,...,g_{n-1}, k)\cdot k^{-1}.
      \]
      The reader may easily check that the cocycle condition $\partial z=0$ indeed implies $\partial c=z$. 
    \end{proof}

\subsubsection{}
Let be given two representations $(\rho_1,\rho_0,\beta,c)$, $(\rho'_1,\rho'_0,\beta',c')$ as in \S~\ref{representation}, on objects $(V_1,V_0)$ and $(V'_1,V'_0)$ of $\mathbf{Ch}'_2(\Aa)$, respectively. It easily follows from the description in \S\ref{1cells} of the 1-cells in $\mathbf{Rep}_{\mathbf{Ch}_2(\Aa)}(\GG)$ that a 1-cell between these representations reduces to the following:

\begin{itemize}
    \item[(I1$'$)] two morphisms of representations $r_i:V_i\to V'_i$, $i=0,1$, which make the diagram
    \[
    \xymatrix{
        \pi_1\ar[r]^{\beta}
             \ar[d]_{\beta'}
        &
        \Aa(V_0,V_1)\ar[d]^{r_{1\ast}}
        \\
        \Aa(V'_0,V'_1)\ar[r]_{r^\ast_0}
        &
        \Aa(V_0,V'_1)
    }
    \]
    commute, and
    \item[(I2$'$)] a map $\mu:\pi_0\to\Aa(V_0,V'_1)$ such that the diagram
    \[
    \xymatrix{
        \pi_0^2\ar[r]^{c}
             \ar[d]_{c'}
        &
        \Aa(V_0,V_1)\ar[d]^{r_{1\ast}}
        \\
        \Aa(V'_0,V'_1)\ar[r]_{r^\ast_0}
        &
        \Aa(V_0,V'_1)
    }
    \]
    commutes up to the boundary map $\partial\mu:\pi_0^2\to\Aa(V_0,V'_1)$ defined by
    \[
    \partial\mu(a,b)=\rho'_1(a)\circ\mu_b-\mu_{a\otimes b}+\mu_a\circ\rho_0(b).
    \]
\end{itemize}
In particular, when $\beta,\beta'$ and $c,c'$ are zero, this simply amounts to two morphisms of representations $r_i:V_i\to V'_i$, $i\in\{0,1\}$, together with a 1-cocycle $\mu:\pi_0\to\Aa(V_0,V'_1)_{\rho_0}^{\rho'_1}$, where $\Aa(V_0,V'_1)_{\rho_0}^{\rho'_1}$ stands for the abelian group $\Aa(V_0,V'_1)$ thought of as a $\pi_0$-bimodule with the left and right actions induced by the representations $\rho'_1$ and $\rho_0$, respectively.

\subsubsection{}
The main theorem of the paper may now be stated as follows.

\begin{theorem}
    Let $\Aa$ be a split $k$-linear abelian category, with $k$ a field of characteristic zero. Then for any finite 2-group $\GG$ the homotopy category of $\mathbf{Rep}_{\mathbf{Ch}_2(\Aa)}(\GG)$ is equivalent to the product category $\Rr ep_\Aa(\pi_0)\times\Rr ep_\Aa(\pi_0)$.
\end{theorem}

    \begin{proof}
        By Proposition~\ref{representations_up_to_equivalence}, we may restrict to representations $(\rho_1,\rho_0,\beta,c)$ with $\beta$ and $c$ equal to zero. In this case, we know that the 1-cells between two such representations $(\rho_1,\rho_0)$ and $(\rho'_1,\rho'_0)$ are given by triples $(r_1,r_0,\mu)$ as before. Now, it follows from \S~\ref{2-cells} that two such 1-cells $(r_1,r_0,\mu)$ and $(s_1,s_0,\nu)$ are 2-isomorphic iff $r_i=s_i$ for $i\in\{0,1\}$, and $\mu,\nu$ differ by the coboundary of some 0-cochain $\omega:1\to\Aa(V_0,V'_1)_{\rho_0}^{\rho'_1}$. The statement now follows from Lemma~\ref{finite_group_cohomology}, which implies that, up to a 2-isomorphism, the 1-cells between such representations are completely given by the pair $(r_1,r_0)$.
    \end{proof}
 
\subsubsection{}
A similar result is expected to hold for {\em compact topological 2-groups} and their {\em topological representations} in a suitable 2-category of {\em topological Baez-Crans 2-vector spaces}. A topological Baez-Crans 2-vector space should reasonably be defined as an object in $\mathbf{Ch}_2(\Aa)$ for some split abelian category $\Aa$ of topological vector spaces over a suitable {\em topological} field. Although the category of all topological vector spaces over an arbitrary topological field is non abelian (images and coimages do not necessarily coincide), it will be so if one restricts to finite dimensional vector spaces over the field of real or complex numbers with the usual topology. Thus Proposition~\ref{zero_morphisms} will also hold in this setting. Moreover, if $\GG$ is a compact topological 2-group, the groups $\pi_0$ and $\pi_1$ will be compact topological groups, and in a topological representation the homomorphism $\beta$ is expected to be continuous. However, there are no compact subgroups in the underlying topological abelian group of a finite dimensional real or complex vector space. Also, the proof of Lemma~\ref{finite_group_cohomology} is expected to work for compact topological groups if one replaces the sum over the elements of $G$ by the corresponding Haar integral, so that both Lemma~\ref{finite_group_cohomology} and Proposition~\ref{representations_up_to_equivalence} are expected to be also true in this topological setting. However, the question would deserve a more careful study. In fact, the theory of topological 2-groups may look different to that of 2-groups. For instance, it is even unclear if any topological 2-group is equivalent to a skeletal topological 2-group because there is no axiom of choice in the category of topological spaces.
